\documentclass[12pt,a4paper]{article}
\title{\bf Exponential Convergence in $L^p$-Wasserstein Distance for Diffusion Processes without Uniformly Dissipative Drift}
\author{Dejun Luo$^a$\footnote{Email: luodj@amss.ac.cn. Partly supported by the Key Laboratory of RCSDS, CAS (2008DP173182), NSFC (11571347) and AMSS (Y129161ZZ1)}
\quad Jian Wang$^b$\footnote{Email: jianwang@fjnu.edu.cn. Partly supported  by NSFC (11201073 and 11522106), NSFFJ (2015J01003) and PNAIA (IRTL1206).}
\vspace{3mm}\\
{\footnotesize $^a$Institute of Applied Mathematics, Academy of Mathematics and Systems Science,}\\
{\footnotesize Chinese Academy of Sciences, Beijing 100190, China}\\
{\footnotesize $^b$School of Mathematics and Computer Science, Fujian Normal University,
Fuzhou 350007, China}}
\date{}

\usepackage{amssymb,amsmath,amsfonts,amsthm,array,color}

\def\R{\mathbb{R}}
\def\E{\mathbb{E}}
\def\P{\mathbb{P}}
\def\Id{\textup{Id}}
\def\C{\mathcal{C}}

\def\ch{{\bf 1}}

\def\d{\textup{d}}

\def\eps{\varepsilon}
\def\<{\langle}
\def\>{\rangle}

\newtheoremstyle{newthm}
 {3pt}
 {3pt}
 {\itshape}
 {}
 {}
 {\textbf{.}}
 {.5em}
 {\thmname{\textbf{#1}}\thmnumber{ \textbf{#2}}\thmnote{ {\textbf{(#3)}}}}

\theoremstyle{newthm}

\newtheorem{theorem}{Theorem}[section]
\newtheorem{lemma}[theorem]{Lemma}       
\newtheorem{corollary}[theorem]{Corollary}
\newtheorem{proposition}[theorem]{Proposition}
\newtheorem{remark}[theorem]{Remark}
\newtheorem{example}[theorem]{Example}

\begin{document}

\maketitle

\makeatletter 
\renewcommand\theequation{\thesection.\arabic{equation}}
\@addtoreset{equation}{section}
\makeatother 

\vspace{-6mm}

\begin{abstract} By adopting the coupling by reflection and
choosing an auxiliary function which is convex near infinity, we
establish the exponential convergence of diffusion semigroups
$(P_t)_{t\ge0}$ with respect to the standard $L^p$-Wasserstein
distance for all $p\in[1,\infty)$. In particular, we show that for
the It\^o stochastic differential equation
  $$\d X_t=\d B_t+b(X_t)\,\d t,$$
if the drift term $b$ satisfies that for any $x,y\in\R^d$,
  $$\langle b(x)-b(y),x-y\rangle\le \begin{cases}
  K_1|x-y|^2,&  |x-y|\le L;\\
  -K_2|x-y|^2,&  |x-y|> L
  \end{cases}$$
holds with some positive constants $K_1$, $K_2$ and $L>0$, then there is a constant $\lambda:=\lambda(K_1,K_2,L)>0$ such that for
all $p\in[1,\infty)$, $t>0$ and $x,y\in\R^d$,
  $$W_p(\delta_x P_t,\delta_y P_t)\leq Ce^{-\lambda t/p}
  \begin{cases}
  |x-y|^{1/p}, & \mbox{if } |x-y|\le 1;\\
  |x-y|, & \mbox{if } |x-y|> 1.
  \end{cases}$$
where $C:=C(K_1,K_2,L,p)$ is a positive constant.
This improves the main result in \cite{Eberle} where the exponential
convergence is only proved for the $L^1$-Wasserstein distance.
\end{abstract}



\section{Introduction}\label{section1}

In this paper we consider the following It\^o stochastic differential equation
  \begin{equation}\label{Ito-SDE}
  \d X_t=\sigma\,\d B_t+b(X_t)\,\d t,
  \end{equation}
where $(B_t)_{t\ge0}$ is a standard $d$-dimensional Brownian motion,
$\sigma\in\R^{d\times d}$ is a non-degenerate constant matrix, and
$b:\R^d\to\R^d$ is a Borel measurable vector field. Recently there
are intensive studies on the existence of the unique strong solution
to \eqref{Ito-SDE} with singular drift $b$. For example, if
$\sigma=c\, \Id$ for some constant $c\neq 0$ and $b$ is bounded and
H\"older continuous, Flandoli et al.\ \cite{FGP10a} proved that
\eqref{Ito-SDE} generates a unique flow of diffeomorphisms on
$\R^d$. The results are recently extended by F.-Y. Wang in
\cite{Wang2014} to (infinite dimensional) stochastic differential
equations with a nice multiplicative noise and a locally
Dini continuous drift. From these results we see that when
the diffusion coefficient $\sigma$ is non-degenerate, quite weak
conditions on $b$ are sufficient to guarantee the well-posedness of
\eqref{Ito-SDE}, which will be assumed throughout this paper.
Moreover, we assume the solution $(X_t)_{t\ge0}$ has finite moments
of all orders.

Denote by $(P_t)_{t\ge0}$ the semigroup associated to
\eqref{Ito-SDE}. If the initial value $X_0$ is distributed as $\mu$,
then for any $t>0$, the distribution of $X_t$ is $\mu P_t$. We are
concerned with the long-time behavior of  $P_t$ as $t$ tends to
$\infty$, more exactly, the rate of convergence to equilibrium of
$\delta_x P_t$ for any $x\in\R^d$. This problem is of fundamental
importance in the study of Markov processes, and has been attacked
by a large number of researchers in the literature. To the authors'
knowledge, there are at least three approaches for obtaining
quantitative ergodic properties. The first one is known as Harris'
theorem which combines Lyapunov type conditions and the notion of
small set, see \cite{Meyn, MT93, Fort05} for systematic
presentations. Recently this method is further developed in
\cite{Hairer, HM2, Bu} with applications to stochastic partial
differential equations (SPDE) and stochastic delay differential
equations (SDDE). The second method employs functional inequalities
to characterize the rate of convergence to equilibrium for
$(P_t)_{t\ge0}$. It is a classical result that for symmetric Markov
processes the Poincar\'e inequality is equivalent to the exponential
decay of the semigroup. More general functional inequalities were
introduced in \cite{Wang00, R-Wang01, Wangbook} to describe
different convergence rates. It was shown in \cite{BCG} that the two
methods above can be linked together by Lyapunov--Poincar\'e
inequalities. We also would like to mention that Bolley et al.
\cite{BGG} recently studied the exponential decay in the
$L^2$-Wasserstein distance $W_2$ via the so-called \emph{WJ}
inequality, by using the explicit formula for time derivative of
$W_2$ along solutions to the Fokker--Planck equation (see
\cite[Theorem 2.1]{BGG}). An application to the granular media was
given in \cite{BGG13}, yielding uniformly exponential convergence to
equilibrium in the presence of non-convex interaction or confinement
potentials.

Yet there is another approach for studying exponential convergence of the semigroup $(P_t)_{t\ge0}$
corresponding to the SDE \eqref{Ito-SDE} considered in this paper, that is, the coupling method.
If the drift vector field $b$ fulfills certain dissipative properties, this latter method
provides explicit rate of convergence to equilibrium in a straightforward way, see e.g. \cite{CG, Eberle11} and
the preprint \cite{Eberle}. The present work is motivated by \cite{Eberle11, Eberle} where the author obtained
exponential decay of $(P_t)_{t\ge0}$ when the  drift $b$ is assumed to be only dissipative at infinity, see
the introduction below for more details.

A good tool for measuring the deviation between probability
distributions is the Wasserstein-type distances which are defined as follows.
Let $\psi\in C^2([0,\infty))$ be a strictly increasing function satisfying $\psi(0)=0$.
Given two probability measures $\mu$ and $\nu$ on $\R^d$, we define the following quantity
  $$W_\psi(\mu,\nu)=\inf_{\Pi\in\C(\mu,\nu)}\int_{\R^d\times\R^d} \psi(|x-y|)\,\d\Pi(x,y),$$
where $|\cdot|$ is the Euclidean norm and $\C(\mu,\nu)$ is the collection of measures on
$\R^d\times\R^d$ having $\mu$ and $\nu$ as marginals. When $\psi$ is concave, the above
definition gives rise to a Wasserstein distance $W_\psi$ in the space $\mathcal P(\R^d)$ of probability
measures $\nu$ on $\R^d$ such that $\int \psi(|z|)\,\nu(\d z)<\infty$.
If $\psi(r)=r$ for all $r\geq 0$, then $W_\psi$ is the standard $L^1$-Wasserstein distance
(with respect to the Euclidean norm $|\cdot|$), which will be denoted by $W_1(\mu,\nu)$
throughout this paper. We are also concerned with the $L^p$-Wasserstein distance
$W_p$ for all $p\in[1,\infty)$, i.e.
  $$W_p(\mu,\nu)=\inf_{\Pi\in\C(\mu,\nu)}\left(\int_{\R^d\times\R^d} |x-y|^p\,\d\Pi(x,y)\right)^{1/p}.$$
Equipped with $W_p$, the totality $\mathcal P_p(\R^d)$ of probability measures having finite moment of order $p$
becomes a complete metric space.

In this paper, we shall establish the exponential convergence of
the map $\mu \mapsto \mu P_t$ with respect to the $L^p$-Wasserstein
distance $W_p$ for all $p\geq 1$. We first recall some known
results.

\begin{theorem}[Uniformly dissipative case]\label{theorem1}
Suppose that $\sigma=\Id$ and there exists a constant $K>0$ such that
  \begin{equation}\label{diss-1}
  \langle b(x)-b(y),x-y\rangle\le -K|x-y|^2\quad \mbox{for all } x,y\in\R^d.
  \end{equation}
Then, for any $p\ge1$ and $t>0$,
  \begin{equation}\label{w2}
  W_p(\mu P_t,\nu P_t)\le e^{-Kt}\,W_p(\mu,\nu)\quad \mbox{for all } \mu,\nu\in\mathcal P_p(\R^d).
  \end{equation}
\end{theorem}

The proof of this result is quite straightforward, by simply using the synchronous coupling
(also called the coupling of marching soldiers in \cite[Example 2.16]{Chen}), see e.g.
\cite[p.2432]{BGG} for a short proof.

In applications, the so-called uniformly dissipative condition \eqref{diss-1} is too strong.
Indeed, it follows from  \cite[Theorem 1]{RS} or \cite[Remark 3.6]{BGG} (also see \cite[Section 3.1.2, Theorem 1]{CG})  that \eqref{w2} holds
for any probability measures $\mu$ and $\nu$ if and only if \eqref{diss-1} holds for all $x$, $y\in\R^d$.
The first breakthrough to get rid of this restrictive condition was done recently by Eberle in \cite{Eberle},
at the price of multiplying a constant $C\geq 1$ on the right hand side of \eqref{w2}.
To state the main result in \cite{Eberle}, we need the following notation which measures
the dissipativity of the drift $b$:
  \begin{equation}\label{kappa}
  \kappa(r):=\sup\bigg\{\frac{\<\sigma^{-1}(x-y),\sigma^{-1}(b(x)-b(y))\>}{2|\sigma^{-1}(x-y)|}:
  x,y\in\R^d \mbox{ with } |\sigma^{-1}(x-y)|=r\bigg\}.
  \end{equation}
As in \cite[(2.3)]{Eberle}, we shall assume throughout the paper that
  $$\int_0^s \kappa^+(r)\,\d r<+\infty\quad \mbox{for all } s>0.$$
This technical condition will be used in Section \ref{section} to
construct the auxiliary function.

\begin{theorem}[{\cite[Corollary 2.3]{Eberle}}]\label{1-thm-2}
Suppose that the vector field $b$ is locally Lipschitz continuous, and there is a constant $c>0$ such that
  \begin{equation}\label{diss-2}
  \kappa(r)\le -cr\quad \mbox{for all } r>0 \mbox{ large enough}.
  \end{equation}
Then there exist positive constants $C$, $\lambda>0$ such that for any $t>0$ and $\mu,\nu\in\mathcal P_1(\R^d)$,
  $$W_1(\mu P_t,\nu P_t)\le Ce^{-\lambda t}W_1(\mu,\nu).$$
\end{theorem}

In particular, when $\sigma=\Id$, the condition \eqref{diss-2} holds true
if \eqref{diss-1} is satisfied only for large $|x-y|$, that is,
  $$\langle b(x)-b(y),x-y\rangle\le -K|x-y|^2, \quad |x-y|\ge L$$
for some  constant $L>0$ large enough. Therefore, \cite[Corollary 2.3]{Eberle} implies that
the map $\mu \mapsto \mu P_t$ converges exponentially with respect to the standard
$L^1$-Wasserstein distance $W_1$ under locally non-dissipative drift, see \cite[Example 1.1]{Eberle}
for more details. The proof of \cite[Corollary 2.3]{Eberle} is based on the coupling by reflection
of diffusion processes and a carefully constructed concave function, cf. \cite[Section 2]{Eberle}.
A number of direct consequences are presented in \cite[Section 2.2]{Eberle} which indicate
that the convergence result as \cite[Corollary 2.3]{Eberle} is extremely useful.

However, \cite[Corollary 2.3]{Eberle} is not satisfactory in the sense that no
information on the $L^2$-Wasserstein distance $W_2$ is provided. This fact has also been
noted in \cite[Section 7.1, Remark 19]{CG}, saying that ``the reflection coupling cannot furnish some
information on $W_2$''. Our main result of this paper shows that this is not the case.

\begin{theorem}\label{main-result}
Assume that there are constants $c>0$ and $\eta\ge 0$ such that for all $r\ge \eta$, one has
  \begin{equation}\label{main-result-1}
  \kappa(r)\leq -cr.
  \end{equation}
For $\eps\in (0,c)$, define
  $$C_0(\eps)=\max\bigg\{\frac{2e^2}\eps \Big(1+\frac2{\sqrt \eps}\Big) \sqrt{\frac2{c-\eps}},\ \frac{2+\sqrt\eps}{\eps(1-e^{-2})} \bigg[\frac{2\sqrt{2} e^2}{\sqrt{\eps(c-\eps)}} + \frac{1}{c-\eps}\bigg]\bigg\}$$
and
  $$\lambda=\frac{\min\{2,2/\eps\}}{C_0(\eps)} \exp\bigg(-\frac c2 \eta^2-\int_0^\eta \kappa^+(s)\,\d s\bigg).$$
Then for any $p>1$, $t>0$ and any $x$, $y\in\R^d$, it holds
  \begin{equation}\label{main-result-2}
  W_p(\delta_x P_t,\delta_y P_t)\leq Ce^{-\lambda t/p}
  \begin{cases}
  |x-y|^{1/p}, & \mbox{if } |x-y|\le 1;\\
  |x-y|, & \mbox{if } |x-y|> 1.
  \end{cases}
  \end{equation}
where $C:=C( c, \eta, \varepsilon, p)>0$ is a positive constant.
\end{theorem}

Theorem \ref{main-result} above does provide new conditions on the
drift term $b$ such that the associated semigroup $(P_t)_{t\ge0}$ is
exponentially convergent with respect to the $L^p$-Wasserstein
distance $W_p$ for all $p\ge 1$. The reason that we can obtain the
exponential convergence in $W_p$ for all $p\ge 1$, not only $W_1$, is
due to our particular choice of the auxiliary function which is
convex near infinity. It is designed by using
Chen--Wang's famous variational formula for the principal eigenvalue of
one-dimensional diffusion operator, see for instance
\cite{Chen-Wang-97} or \cite[Section 3.4]{Chen}. The reader can
refer to \cite{Chen12} and the references therein for recent studies
on this topic.

The assertion of Theorem \ref{main-result} can be slightly strengthened if \eqref{main-result-1} is replaced by a stronger condition.

\begin{theorem}\label{thm-theta}
Assume that there are constants $c>0$, $\eta>0$ and $\theta>1$ such that for all $r\ge \eta$, one has
  \begin{equation}\label{thm-theta-1}
  \kappa(r)\leq -cr^\theta.
  \end{equation}
Let $\lambda$ be defined as in Theorem $\ref{main-result}$ with $c$ replaced by $c\eta^{\theta-1}$. Then there is a positive constant $C$ such that for all $t>0$ and $x,y\in\R^d$, it holds
  \begin{equation}\label{thm-theta-2}
  W_p(\delta_x P_t,\delta_y P_t)\leq Ce^{-\lambda t/p}
  \begin{cases}
  |x-y|^{1/p}, & \mbox{if } |x-y|\le 1;\\
  |x-y|\wedge \frac{1}{t\wedge 1}, & \mbox{if } |x-y|> 1.
  \end{cases}
  \end{equation}
\end{theorem}

The idea of the proof is to use synchronous coupling for large $|x-y|$ and the coupling by reflection for small $|x-y|$. For the latter part, we can directly use the result of Theorem \ref{main-result}, since \eqref{thm-theta-1} implies that \eqref{main-result-1} holds with $c\eta^{\theta-1}$ if $\eta>0$.

Before presenting the consequences of Theorem \ref{main-result}, let
us first make some comments and give some examples. In the beginning, we intended to generalize Eberle's results by
assuming that \emph{there is a constant $c>0$ such that}
  \begin{equation}\label{1-rem-1}
\kappa(r)\leq -c \quad \it{for\,all }\,r\, \it{ large\, enough} .
  \end{equation}
\noindent It turns out that under mild conditions on $\kappa$, the two conditions \eqref{diss-2} and
\eqref{1-rem-1} are equivalent, up to changing the constants. More explicitly, we have

\begin{proposition}\label{2-prop-1}
Assume that there are constants $c,r_0>0$ such that $\kappa(r_0)\leq -c$ and $\delta_0:=\sup_{0\leq r\leq r_0}\kappa(r)<+\infty$. Then, the condition \eqref{diss-2} holds with some other positive constant.
\end{proposition}

This result indicates that if the function $\kappa$ is locally
bounded from above, then the following statements are equivalent:
\begin{itemize}
\item[(i)] there exist constants $c,r_0>0$ such that $\kappa(r_0)\leq -c$;
\item[(ii)] there exist constants $c>0$ and $\theta\le 1$ such that $\kappa(r)\le -cr^{\theta}$ for $r>0$ large enough;
\item[(iii)] there exists a constant $c>0$ such that $\kappa(r)\le -cr$ for  $r>0$ large enough.
\end{itemize}
Note that, none of the above condition is equivalent to that there exist constants $c>0$ and $\theta>1$
such that for $r>0$ large enough, $\kappa(r)\le -cr^{\theta}$, see e.g.\ $b(x)=-x$ for all $x\in\R^d$. Compared to \eqref{diss-2}, the seemingly much weaker condition (i), i.e.\ \emph{there exist two constants $c,r_0>0$ such that $\kappa(r_0)\le -c$}, has the obvious advantage of being easily verifiable. Thus we shall sometimes
use this formulation in the sequel.

The equivalence stated above also indicates that Theorem
\ref{main-result} is sharp in some situation, as shown by the next
example.

\begin{example}\label{1-example-1}
Assume that $\sigma=\Id$ and $b(x)=\nabla V(x)$ with
$V(x)=-(1+|x|^2)^{\delta/2}$ for some $\delta\in(0,2)$. Then, we
have the following statements.
\begin{itemize}
\item[$(1)$] If $\delta\in(0,1)$, then $\kappa(r)\geq 0$ for all $r$ large
enough, and the inequality \eqref{main-result-2} does not hold for
any positive constants $C$ and $\lambda$ with $p=1$.
\item[$(2)$] If $\delta\in[1,2)$, then $\kappa(r)=0$ for all $r\geq
0$, and so for all $x,$ $y\in\R^d$ and $t>0$,
  \begin{equation}\label{1-example-1.1}
  W_1(\delta_x P_t,\delta_y P_t)\leq |x-y|.
  \end{equation}
On the other hand, it holds that
  \begin{equation}\label{1-example-1.2}
  d_{TV}(\delta_x P_t,\delta_y P_t)\leq \sqrt{\frac{2}{\pi t}}\, |x-y|
  \quad \mbox{for all }t\geq 0 \mbox{ and } x,y\in\R^d,
  \end{equation}
where $d_{TV}$ is the total variation distance between probability measures.
\end{itemize}
\end{example}

To show the power of Theorem \ref{thm-theta}, we consider the
following example which yields the exponential convergence of the
semigroup $(P_t)_{t\ge0}$ with respect to the $L^p$-Wasserstein
distance $W_p$ $(p>2)$ for super-convex potentials. The assertion
below improves the results mentioned in \cite[Section 6.1]{CG}.

\begin{example}\label{1-example-2} \rm
\emph{Let $\sigma=\Id$ and $b(x)=\nabla V(x)$ with $V(x)=-|x|^{2\alpha}$ and $\alpha>1$.
It follows from} \cite[Section 6, Example 1]{CG} \emph{that there
is a constant $c>0$ such that for all $r>0$,
  \begin{equation}\label{1-example-2.1}
  \kappa(r)\le -c r^{2\alpha-1}.
  \end{equation}
Then, according to Theorem $\ref{thm-theta}$, the associated semigroup $(P_t)_{t\ge0}$
converges exponentially with respect to the $L^p$-Wasserstein distance $W_p$ for any $p\geq 1$.  More explicitly, there is a constant $\lambda:=\lambda(\alpha)>0$ such that for any $p\ge1$,
$x$, $y\in\R^d$ and $t>0$,
  $$W_{p}(\delta_x P_t,\delta_y P_t)\leq Ce^{-\lambda t}\left[|x-y|^{1/p}{\bf 1}_{\{|x-y|\le 1\}}
  + \Big(|x-y| \wedge \frac{1}{t\wedge1}\Big){\bf 1}_{\{|x-y|\ge 1\}}\right] $$ holds with some constant $C:=C(\alpha,p)>0$.}

\emph{Note that \eqref{1-example-2.1} implies that  for all $x$, $y\in\R^d$,
  $$\langle b(x)-b(y),x-y\rangle\le -2c |x-y|^{2\alpha}.$$ Therefore,
the uniformly dissipative condition \eqref{diss-1} fails when $x,y\in\R^d$ are sufficiently close to each other.
That is, one cannot deduce directly from Theorem $\ref{theorem1}$ the exponential
convergence with respect to the $L^p$-Wasserstein distance $W_p\ (p\geq 1)$.}
\end{example}

As applications of Theorem \ref{main-result}, we consider the
existence and uniqueness of the invariant probability measure, and
also the exponential convergence of the semigroup with respect to
the $L^p$-Wasserstein distance $W_p$.  For $p\in(1,\infty)$, we
define
  $$\phi_p(r)=\begin{cases}
  r^{1/p}, & \mbox{if } r< p^{-p/(p-1)};\\
  r-p^{-p/(p-1)}+p^{-1/(p-1)}, & \mbox{if } r\ge p^{-p/(p-1)}.
  \end{cases}$$
Finally, let $\phi_1(r)=r$ for all $r\geq 0$. Then for all
$p\in[1,\infty)$, $\phi_p$ is a concave $C^1$-function on $\R_+$,
thus $W_{\phi_p}$ is a well defined distance on $\mathcal
P_p(\R^d)$. Moreover,
  $$r\vee r^{1/p}\leq \phi_p(r)\leq r+r^{1/p}\le 2 (r\vee r^{1/p}) \quad \mbox{for all } r\geq 0,\, p\in[1,\infty).$$

\begin{corollary}\label{invariant}
Suppose that the drift term $b$ is locally bounded on $\R^d$, and
\eqref{main-result-1} holds for all $r>0$ large enough with some
constant $c>0$. Then, there exists a unique
invariant probability measure $\mu\in \cap_{p\geq 1}\mathcal{P}_p(\R^d)$,
such that there is a constant $\lambda:=\lambda(c)>0$ such that for all $p\in[1,\infty)$ and for any probability measure
$\nu\in \mathcal P_p(\R^d)$,
  $$W_p(\nu P_t,\mu)\leq Ce^{-\lambda t}W_{\phi_p}(\nu, \mu),\quad t\geq 0$$
holds with some positive constant $C:=C(c,p)$.
\end{corollary}

\begin{remark}\label{remark-ref}\rm \begin{itemize}\item[(1)] Under the assumptions of Corollary \ref{invariant}, it is easy to establish the following Foster-Lyapunov type conditions:
$$L \phi(x)\le - c_1\phi(x)+c_2,\quad x\in \R^d,$$ where $L$ is the
generator of the underlying diffusion process, $\phi(x)=|x|^2$ and
$c_1,c_2$ are two positive constants. Due to the existence and
uniqueness of the invariant probability measure, we know that the
diffusion process exponentially converges to the unique invariant
probability measure $\mu$ with respect to the total variation
distance. That is, there are a constant $\theta>0$ and a measurable
function $C(x)$ such that for all $x\in \R^d$ and $t>0$,
$$  d_{TV}(P(t,\cdot),\pi)\le C(x)e^{-\theta t},$$ where $P(t,\cdot)$ is the associated transition probability.

\item[(2)] As was pointed by the referee, the conclusion of
Corollary \ref{invariant} for $p=1$ also could be deduced from
Theorem \ref{thm-theta} and the Foster-Lyapunov type condition
above, by Harris's theorem for the exponential convergence to the
invariant measure in the Wasserstein metric. See \cite[Theorem
4.8]{HM2} and \cite[Theorem 2.4]{Bu} for more details.
\end{itemize}
\end{remark}

The following statement is concerned with symmetric diffusion
processes. Though we believe the assertion below is known (see e.g.
\cite[Corollary 1.4]{Wang99}), we stress the relation between the
exponential convergence with respect to $L^1$-Wasserstein distance
$W_1$ and that with respect to the $L^2$-norm, which is equivalent
to the Poincar\'{e} inequality.

\begin{corollary}\label{poincare}
Let $U$ be a $C^2$-potential defined on $\R^d$ such that its Hessian matrix $\textup{Hess}(U)\geq -K$
for some $K>0$. Assume that $\mu(\d x)=e^{-U(x)}\,\d x$ is a probability measure on $\R^d$. If there
exists a constant $L>0$ such that
  \begin{equation}\label{poincare-1}
  \inf_{|x-y|=L}\big\langle \nabla U(x)-\nabla U(y), x-y\big\rangle>0,
  \end{equation}
then $\mu$ satisfies the Poincar\'{e} inequality, i.e.
  \begin{equation}\label{poincare-2}
  \mu(f^2)-\mu(f)^2\le C\int |\nabla f(x)|^2\,\d \mu,\quad f\in C_c^2(\R^d)
  \end{equation}
holds for some constant $C>0$.
\end{corollary}

\section{Preliminaries}\label{section}
\subsection{Coupling by Reflection}
Similar to the main result in \cite{Eberle}, the proof of Theorem
\ref{main-result} is based on the reflection coupling of Brownian
motion, which was introduced by Lindvall and Rogers \cite{LR} and
developed by Chen and Li \cite{CL}. First, we give a brief
introduction of the coupling by reflection. Together with
\eqref{Ito-SDE}, we also consider
  \begin{equation}\label{coupling}
  \d Y_t=\sigma(\Id-2e_te_t^\ast)\,\d B_t+b(Y_t)\,\d t, \quad t<T,
  \end{equation}
where $\Id\in\R^{d\times d}$ is the identity matrix,
  $$e_t=\frac{\sigma^{-1}(X_t-Y_t)}{|\sigma^{-1}(X_t-Y_t)|}$$
and $T=\inf\{t>0:X_t=Y_t\}$ is the coupling time. For $t\geq T$, we
shall set $Y_t=X_t$. Then, the process $(X_t,Y_t)_{t\ge0}$ is called
the coupling by reflection of $(X_t)_{t\ge0}$. Under our assumptions, the refection coupling
$(X_t,Y_t)_{t\ge0}$ can be realized as a non-explosive diffusion
process in $\R^{2d}$. The difference process
$(Z_t)_{t\ge0}=(X_t-Y_t)_{t\ge0}$ satisfies
  \begin{equation}\label{difference}
  \d Z_t=\frac{2Z_t}{|\sigma^{-1}Z_t|}\,\d W_t+(b(X_t)-b(Y_t))\,\d t,\quad t<T,
  \end{equation}
where $(W_t)_{0\leq t<T}$ is a one dimensional Brownian motion expressed
by $W_t=\int_0^t \<e_s,\d B_s\>$.

Next, by It\^o's formula and \eqref{difference}, for $t<T$,
  \begin{align*}
  \d \big(|\sigma^{-1}Z_t|^2\big)&=2\<\sigma^{-1}Z_t,\sigma^{-1}\d Z_t\>+\<\sigma^{-1}\d Z_t,\sigma^{-1}\d Z_t\>\\
  &=4|\sigma^{-1}Z_t|\,\d W_t+2\big\<\sigma^{-1}Z_t,\sigma^{-1}(b(X_t)-b(Y_t))\big\>\,\d t
  +4\,\d t.
  \end{align*}
Denote by $r_t=|\sigma^{-1}Z_t|$. Then
  \begin{equation}\label{difference.1}
  \begin{split}
  \d r_t&=\frac{1}{2r_t}\,\d r_t^2-\frac1{8r_t^3}\,\d r_t^2\cdot\d r_t^2\\
  &=2\,\d W_t+\frac1{r_t}\big\<\sigma^{-1}Z_t,\sigma^{-1}(b(X_t)-b(Y_t))\big\>\,\d t.
  \end{split}
  \end{equation}

  \subsection{Auxiliary Function}
  For any $\varepsilon\in(0,c)$, let $\psi\in C^2([0,\infty))$ be the
following strictly increasing function
 \begin{equation}\label{af}\psi(r)=\int_0^r \exp\left({-\int_0^s \kappa^*(v)\,\d v}\right)
  \bigg\{\int_s^\infty \exp\left({\int_0^u\big[\kappa^*(v)+\varepsilon v\big]\,\d v}\right)\,\d u\bigg\}\, \d s,\end{equation}
where
  $$\kappa^*(r)=\begin{cases}
  \kappa^+(r), & \mbox{if } 0<r\le \eta;\\
  -cr, & \mbox{if } r>\eta.
  \end{cases}$$

\ \

\begin{remark}\label{rem-psi} \rm
The definition of $\psi$ seems a little bit strange at first sight, indeed, it is motivated by Chen--Wang's variational
formula for principal eigenvalue of one-dimensional diffusion operator
  $$Lf(r)=f''(r)+b(r)f'(r)$$
on $[0,\infty)$. One key to this famous formula is the following
elegant mimic eigenfunction $g$ associated with the first eigenvalue
$($see \cite[pp.\ 52--53]{Chen} for a heuristic argument$)$:
  $$g(r)=\int_0^r e^{-C(s)}\,\d s\int_s^\infty h(u) e^{C(u)}\,\d u,$$
where
  $$C(s)=\int_0^s b(u)\,\d u,\quad h(s)=\left(\int_0^r e^{-C(s)}\,\d s\right)^{1/2}.$$
Now, let $b(r)=-cr$ in the definition of the function $g$ above. Then we have
$$g(r)\asymp\psi(r)\quad\mathrm{ as  }\,\,r\to \infty$$ for some proper choice of the constant $\varepsilon$ in the definition of $\psi$.
\end{remark}

 \ \

On the one hand, it is easy to see that for any $r>0$,
  \begin{equation*}
  \begin{split}
  \psi'(r)=&\exp\left({-\int_0^r \kappa^*(v)\,\d v}\right)
  \int_r^\infty \exp\left(\int_0^u\big[\kappa^*(v)+\varepsilon v\big]\,\d v\right)\,\d u,\\
  \psi''(r)=&-\kappa^*(r)\psi'(r)- \exp\big(\eps r^2/2\big),
  \end{split}
  \end{equation*}
and so
  \begin{equation}\label{proof.4}
  \psi''(r)+\kappa(r)\psi'(r)\le - \exp\big(\eps r^2/2\big),\quad r\ge0.
  \end{equation}
On the other hand, for all $r>0$,
  $$\kappa^*(r)= [\kappa^+(r)+cr ]{\bf 1}_{[0,\eta]}(r)-cr.$$
Thus, for all $r>0$,
  \begin{equation}\label{r1}
  \begin{split}
  \psi(r)&=\int_0^r\exp\left(\frac{c}2 s^2-\int_0^{\eta\wedge s} [\kappa^+(v)+cv ]\,\d v\right)\\
  &\quad\times \bigg\{\int_s^\infty\exp\left(-\frac{c-\varepsilon}{2}u^2 +\int_0^{u\wedge \eta} [\kappa^+(v)+cv ]\,\d v\right)\,\d u\bigg\} \d s \\
  &\le \exp\bigg(\int_0^{\eta} [\kappa^+(v)+cv ]\,\d v \bigg)
  \int_0^r e^{cs^2/2} \bigg(\int_s^\infty e^{-(c-\eps)u^2/2}\, \d u\bigg) \d s.
  \end{split}
  \end{equation}

Define
  \begin{align*}
  \psi_1(r)&=\int_0^r e^{cs^2/2} \bigg(\int_s^\infty e^{-(c-\eps)u^2/2}\, \d u\bigg) \d s,\quad
  \psi_2(r)=\frac{e^{\eps r^2/2} -1}{r (1+r)}.
  \end{align*}
We first show that $\psi_1(r)$ and $\psi_2(r)$ are comparable.

\begin{lemma}\label{2-lem-1}
There exist two constants $C_0:=C_0(\eps)$ and $\hat C_0:=\hat C_0(\eps)>0$ such that for all $r\ge0$,
  $$\hat C_0\psi_2(r)\leq \psi_1(r)\leq C_0 \psi_2(r).$$
\end{lemma}

\begin{proof}
Note that $r(1+r)\sim r$ as $r\to 0$ and $r(1+r)\sim r^2$ as $r\to \infty$, where $\sim$ means the two quantities are of the same order. By L'H\^opital's law,
  \begin{align*}
  \lim_{r\to 0} \frac{\psi_1(r)}{\psi_2(r)}&= \lim_{r\to 0} \frac{e^{cr^2/2} \int_r^\infty e^{-(c-\eps)u^2/2}\, \d u}{-r^{-2} \big(e^{\eps r^2/2}-1\big)+r^{-1}e^{\eps r^2/2} \eps r}\\
  &=\frac2\eps \int_0^\infty e^{-(c-\eps)u^2/2}\, \d u \\
  &=\frac2\eps \sqrt{\frac\pi{2(c-\eps)}}.
  \end{align*}
Next, using L'H\^opital's law twice,
  \begin{align*}
  \lim_{r\to \infty} \frac{\psi_1(r)}{\psi_2(r)}&=\lim_{r\to \infty} \frac{\psi_1(r)}{r^{-2} e^{\eps r^2/2}}\\
  &= \lim_{r\to \infty} \frac{e^{cr^2/2} \int_r^\infty e^{-(c-\eps)u^2/2}\, \d u} {e^{\eps r^2/2} [-2r^{-3}+\eps r^{-1}]}\\
  &=\lim_{r\to \infty} \frac{\int_r^\infty e^{-(c-\eps)u^2/2}\, \d u}{e^{-(c-\eps) r^2/2} r^{-2} [-2r^{-1}+\eps r]}\\
  &=-\lim_{r\to \infty} \frac{1}{h(r)},
  \end{align*}
where
  \begin{align*}
  h(r)&= -(c-\eps)r^{-1} (-2r^{-1}+\eps r)-2{r^{-3}} (-2r^{-1}+\eps r) +r^{-2} (2r^{-2}+\eps).
  \end{align*}
Thus,
  $$\lim_{r\to \infty} \frac{\psi_1(r)}{\psi_2(r)}=\frac1{\eps(c-\eps)}.$$
Therefore, the required assertion follows from the two limits above.
\end{proof}

By \eqref{r1} and Lemma \ref{2-lem-1}, we have
  \begin{equation*}
  \psi(r)\leq C_0\exp\bigg(\int_0^{\eta} [\kappa^+(v)+cv ]\,\d v \bigg) \frac{e^{\eps r^2/2} -1}{r (1+r)}
  \end{equation*}
and
  \begin{equation*}
  \psi(r)\geq \hat C_0 \exp\bigg(-\int_0^{\eta}\! \big[\kappa^+(v)+cv \big]\,\d v \bigg)\frac{e^{\eps r^2/2} -1}{r (1+r)}.
  \end{equation*}
Roughly speaking, the above two inequalities imply that the auxiliary function $\psi$ behaves like $c'r$ for small $r$, and grows exponentially fast as $e^{c''r^2}$ for large $r$. Hence, the function $\psi(r)$ can be used to control the function $r^p$ with $p\ge 1$. More explicitly, we have

\begin{corollary} There is a constant $C_1>1$ such that for all $r\ge0$,
  \begin{equation}\label{2-estimate-2-1}
  C_1^{-1}\big[r \vee (e^{\eps r^2/2}-1)\big]\le \psi(r)\le C_1\big[r \vee (e^{\eps r^2/2}-1)\big].
  \end{equation}
Consequently, for any $p\ge1$, there is a constant $C_2=C_2(p,\eps)>0$ such that for all $r\geq 0$,
   \begin{equation}\label{2-estimate-2} r^p\leq C_2 \psi(r).  \end{equation}
\end{corollary}

Furthermore, we can give an explicit estimate to the constant $C_0$ in Lemma $\ref{2-lem-1}$, which will be used in the exponential convergence rate.

\begin{lemma}\label{2-lem-2}
The constant $C_0$ in Lemma $\ref{2-lem-1}$ has the following expression:
  $$C_0=\max\bigg\{\frac{2e^2}\eps \Big(1+\frac2{\sqrt \eps}\Big) \sqrt{\frac2{c-\eps}},\ \frac{2+\sqrt\eps}{\eps(1-e^{-2})} \bigg[\frac{2\sqrt{2} e^2}{\sqrt{\eps(c-\eps)}} + \frac{1}{c-\eps}\bigg]\bigg\}.$$
\end{lemma}

\begin{proof}
In order to estimate $\psi_1(r)$, we need the following inequality on the tail of standard Gaussian distribution (e.g.\ see \cite[(3)]{Lutz}):
  \begin{equation*}\label{gaussian-tail}
  1-\Phi(r)\leq \frac{2\phi(r)}{\sqrt{2+r^2}+r}\quad \mbox{for all } r>0,
  \end{equation*}
where $\Phi(r)$ and $\phi(r)$ are respectively the distribution and density function of the standard Gaussian distribution $N(0,1)$. Consequently, for $s>0$,
  \begin{align*}
  \int_s^\infty e^{-(c-\eps)u^2/2}\, \d u&=\frac1{\sqrt{c-\eps}}\int_{\sqrt{c-\eps}\, s}^\infty e^{-v^2/2}\,\d v\leq \frac1{\sqrt{c-\eps}}\cdot \frac{2e^{-(c-\eps)s^2/2}}{\sqrt{2+(c-\eps)s^2}+\sqrt{c-\eps}\, s}.
  \end{align*}
Substituting this estimate into the expression of $\psi_1$ leads to
  \begin{equation}\label{2-lem-2.1}
  \psi_1(r)\leq \frac2{\sqrt{c-\eps}}\int_0^r \frac{e^{\eps s^2/2}}{\sqrt{2+(c-\eps)s^2}+\sqrt{c-\eps}\, s}\,\d s=: \frac2{\sqrt{c-\eps}}\tilde\psi_1(r).
  \end{equation}

Next, we consider two cases. (i) If $r\leq 2/\sqrt{\eps}$, then
  \begin{equation*}\label{2-lem-2.2}
  \tilde\psi_1(r)\leq \int_0^r \frac{e^2}{\sqrt{2}}\,\d s=\frac{e^2}{\sqrt 2}\, r.
  \end{equation*}
(ii) If $r>2/\sqrt{\eps}$, then
  \begin{equation}\label{2-lem-2.3}
  \begin{split}
  \tilde \psi_1(r)&=\bigg(\int_0^{2/\sqrt\eps}+\int_{2/\sqrt\eps}^r \bigg) \frac{e^{\eps s^2/2}}{\sqrt{2+(c-\eps)s^2}+\sqrt{c-\eps}\, s}\,\d s\\
  &\leq e^2\sqrt{\frac2\eps} +\frac1{2\sqrt{c-\eps}}\int_{2/\sqrt\eps}^r \frac{e^{\eps s^2/2}}{s}\,\d s.
  \end{split}
  \end{equation}
By the integration by parts formula,
  \begin{align*}
  \int_{2/\sqrt\eps}^r \frac{e^{\eps s^2/2}}{s}\,\d s&=\frac1\eps \int_{2/\sqrt\eps}^r \frac{\d\big( e^{\eps s^2/2}\big)}{s^2}\\
  &=\frac1\eps \bigg[\frac{e^{\eps r^2/2}}{r^2}-\frac\eps 4 e^2 +\int_{2/\sqrt\eps}^r \frac{e^{\eps s^2/2}}{s^3}\,\d s\bigg]\\
  &\leq \frac{e^{\eps r^2/2}}{\eps r^2} +\frac12 \int_{2/\sqrt\eps}^r \frac{e^{\eps s^2/2}}{s}\,\d s.
  \end{align*}
Therefore,
  $$\int_{2/\sqrt\eps}^r \frac{e^{\eps s^2/2}}{s}\,\d s\leq \frac{2e^{\eps r^2/2}}{\eps r^2}.$$
Substituting this estimate into \eqref{2-lem-2.3} yields
  $$\tilde \psi_1(r)\leq e^2\sqrt{\frac2\eps} +\frac{e^{\eps r^2/2}}{\eps\sqrt{c-\eps}\, r^2}.$$
Summarizing the above two cases and using \eqref{2-lem-2.1}, we obtain
  \begin{equation}\label{2-lem-2.4}
  \psi_1(r)\leq \begin{cases}
  \sqrt{\frac2{c-\eps}}e^2 r, & \mbox{if } r\leq \frac2{\sqrt\eps};\\
  \frac{2\sqrt{2} e^2}{\sqrt{\eps(c-\eps)}}+\frac{2e^{\eps r^2/2}}{\eps(c-\eps)r^2}, & \mbox{if } r> \frac2{\sqrt\eps}.
  \end{cases}
  \end{equation}

Furthermore, since $$e^{\eps r^2/2}-1\geq \eps r^2/2,\quad r\ge0,$$  it is easy to show that for all $r\in[0,2/\sqrt{\eps}\,]$, it holds
  \begin{equation}\label{2-lem-2.5}
  \sqrt{\frac2{c-\eps}}e^2 r\leq \frac{C_3}{r(1+r)}\big(e^{\eps r^2/2}-1\big),
  \end{equation}
where
  $$C_3=\frac{2e^2}\eps \Big(1+\frac2{\sqrt \eps}\Big) \sqrt{\frac2{c-\eps}}.$$
On the other hand, for $r>2/\sqrt{\eps}$, we have
  \begin{align*}
  r(1+r)\bigg[\frac{2\sqrt{2} e^2}{\sqrt{\eps(c-\eps)}}+\frac{2e^{\eps r^2/2}}{\eps(c-\eps)r^2}\bigg]
  &= \Big(1+\frac{1}{r}\Big)\bigg[\frac{2\sqrt{2} e^2}{\sqrt{\eps(c-\eps)}} r^2+\frac{2e^{\eps r^2/2}}{\eps(c-\eps)}\bigg]\\
  &\leq \frac{2+\sqrt\eps}{\eps}\bigg[\frac{2\sqrt{2} e^2}{\sqrt{\eps(c-\eps)}} + \frac{1}{c-\eps}\bigg]e^{\eps r^2/2}
  \end{align*}
and
  $$e^{\eps r^2/2}-1\geq (1-e^{-2})e^{\eps r^2/2}.$$
Combining the above two inequalities, we deduce that for all $r>2/\sqrt{\eps}$,
  \begin{equation}\label{2-lem-2.6}
  \frac{2\sqrt{2} e^2}{\sqrt{\eps(c-\eps)}}+\frac{2e^{\eps r^2/2}}{\eps(c-\eps)r^2}\leq \frac{C_4}{r(1+r)}\big(e^{\eps r^2/2}-1\big).
  \end{equation}
with
  $$C_4=\frac{2+\sqrt\eps}{\eps(1-e^{-2})}\bigg[\frac{2\sqrt{2} e^2}{\sqrt{\eps(c-\eps)}} + \frac{1}{c-\eps}\bigg].$$
Having the two inequalities \eqref{2-lem-2.5} and \eqref{2-lem-2.6} in hand, and using \eqref{2-lem-2.4}, we complete the proof.
\end{proof}

We also need the following simple result.

\begin{lemma}\label{2-lem-3}
For all $r>0$,
  \begin{equation}\label{psi-2}
  e^{\eps r^2/2}\geq \bar C_0\psi_2(r),
  \end{equation} where $$\bar C_0=\min\{2,2/\varepsilon\}.$$
\end{lemma}

\begin{proof}
It is obvious that if $r\geq 1$, then
  $$\psi_2(r)=\frac{e^{\eps r^2/2}-1}{r(1+r)}\leq \frac{e^{\eps r^2/2}}2.$$
If $0\leq r\leq 1$, then, using the fact that
$$e^{r}-1\le re^{r}, \quad r\ge 0,$$ we have
  $$\psi_2(r)=\frac{e^{\eps r^2/2}-1}{r(1+r)}\leq \frac{\eps r e^{\eps r^2/2} }{2}\leq \frac{\eps  e^{\eps r^2/2} }{2}. $$
The required assertion follows immediately from these two conclusions.
\end{proof}

Finally, we present a consequence of all the previous results in this part.
\begin{corollary} Let $\psi$ be the function defined by \eqref{af}, and $\lambda$ be the constant in Theorem $\ref{main-result}$. Then, for all $r>0$,
\begin{equation}\label{af-con}\psi''(r)+\kappa(r)\psi'(r)\le -\lambda \psi(r).\end{equation} \end{corollary}

\begin{proof}
 By \eqref{proof.4}, we deduce from Lemmas \ref{2-lem-2} and \ref{2-lem-3} that
  \begin{align*}
  \psi''(r)+\kappa(r)\psi'(r)&\le -e^{\eps r^2/2}\le  -\bar C_0 \psi_2(r)\leq -\frac{\bar C_0}{C_0}\psi_1(r)\\
  &\leq -\frac{\bar C_0}{C_0} \exp\bigg(-\frac c2 \eta^2-\int_0^\eta \kappa^+(s)\,\d s\bigg) \psi(r),
  \end{align*}
where the last inequality follows from \eqref{r1}. This, along with the definition of the constant $\lambda$ in Theorem \ref{main-result}, yields the required assertion.\end{proof}

\section{Proofs of Theorems and Proposition} \label{section2}
We first give the

\begin{proof}[Proof of Theorem $\ref{main-result}$]Recall that we assume the solution $(X_t)_{t\ge0}$ to
\eqref{Ito-SDE} has finite moments of all orders. In particular, the
left hand side of \eqref{main-result-2} is finite for any $x$,
$y\in\R^d$ and $t>0$.

Let $\psi$ be the function defined by \eqref{af}. According to \eqref{difference.1} and  It\^o's formula, it holds that
  \begin{equation*}\label{proof.3}
  \begin{split}
  \d\psi(r_t)&=2\psi'(r_t)\,\d W_t+2\bigg[\psi''(r_t)+\frac{\psi'(r_t)}{2r_t}
  \big\<\sigma^{-1}Z_t,\sigma^{-1}(b(X_t)-b(Y_t))\big\>\bigg]\d t\\
  &\leq 2\psi'(r_t)\,\d W_t+2\big[\psi''(r_t)+\kappa(r_t)\psi'(r_t)\big]\,\d t.
  \end{split}
  \end{equation*}
Combining this inequality with \eqref{af-con}, we obtain
  \begin{equation}\label{3-proof.1}
  \begin{split}
  \d\psi(r_t)\leq 2\psi'(r_t)\,\d W_t-\lambda \psi(r_t)\,\d t,
  \end{split}
  \end{equation} where $\lambda$ is the constant in Theorem \ref{main-result}.

For $n\geq 1$, define the stopping time
  $$T_n=\inf\{t>0: r_t\notin [1/n, n]\}.$$
Then, for $t\leq T_n$, the inequality \eqref{3-proof.1} yields
  $$\d\big[e^{\lambda t}\psi(r_t)\big]\leq 2e^{\lambda t}\psi'(r_t)\,\d W_t.$$
Therefore,
  $$e^{\lambda (t\wedge T_n)}\psi(r_{t\wedge T_n})\leq \psi(r_0)
  +2\int_0^{t\wedge T_n} e^{\lambda s}\psi'(r_s)\,\d W_s.$$
Taking expectation in the both hand sides of the inequality above leads to
  $$\E\big[e^{\lambda (t\wedge T_n)}\psi(r_{t\wedge T_n})\big]\leq \psi(r_0).$$
Since the coupling process $(X_t,Y_t)_{t\ge0}$ is non-explosive, we have $T_n\uparrow T$ a.s. as $n\to\infty$,
where $T$ is the coupling time. Thus by Fatou's lemma, letting $n\to\infty$ in the above inequality
gives us
  \begin{equation}\label{3-proof.2}
  \E\big[e^{\lambda (t\wedge T)}\psi(r_{t\wedge T})\big]\leq \psi(r_0).
  \end{equation}
Thanks to our convention that $Y_t=X_t$ for $t\geq T$, we have $r_t=0$ for all $t\geq T$.
Therefore,
  $$\E\big[e^{\lambda (t\wedge T)}\psi(r_{t\wedge T})\big]
  =\E\big[e^{\lambda t}\psi(r_t)\ch_{\{T>t\}}\big]=\E[e^{\lambda t}\psi(r_t)].$$
Combining this with \eqref{3-proof.2}, we arrive at
  $$\E\psi(r_t)\leq \psi(r_0)e^{-\lambda t}.$$
That is,
  \begin{equation} \label{r3}
  \E\psi(|\sigma^{-1}(X_t-Y_t)|)\leq \psi(|\sigma^{-1}(x-y)|)e^{-\lambda t}.
  \end{equation}

If $|\sigma^{-1}(x-y)|\le \eta$, then for any $p\ge 1$ and any $t>0$, we deduce from \eqref{2-estimate-2}, \eqref{r3} and \eqref{2-estimate-2-1} that
  \begin{equation}\label{small}
  \begin{split}
  \E\big(|\sigma^{-1}(X_t-Y_t)|^p\big)&\leq C_2\, \E\psi\big(|\sigma^{-1}(X_t-Y_t)|\big)
  \le C_5 e^{-\lambda t}|\sigma^{-1}(x-y)|.
  \end{split}
  \end{equation}
It is clear that
  $$C_6^{-1}|z|\le |\sigma^{-1}z|\le C_6|z|,\quad z\in\R^d$$
for some constant $C_6>1$. Therefore, if $|x-y|\le \eta/C_6$,
then for any $p\ge1$ there is a constant $C_7>0$ such that
  \begin{equation*}
  \E|X_t-Y_t|^p\le C_7 e^{-\lambda t}|x-y|,\quad t>0,
  \end{equation*}
which implies that for any $p\ge1$ and any $x,y\in\R^d$ with $|x-y|\le \eta/C_6$,
  \begin{equation}\label{small-1}
  W_{p}(\delta_x P_t,\delta_y P_t)\le C_7^{1/p}e^{-\lambda t/p}|x-y|^{1/p}.
  \end{equation}

Now for any $x,y\in\R^d$ with $|x-y|>\eta/C_6$, take
$n:=\big[C_6|x-y|/\eta\big]+ 1\geq 2$. We have
  \begin{equation}\label{proof.2}\begin{split}
 \frac{n}{2}\le  n-1\leq \frac{C_6|x-y|}{\eta}\le n.  \end{split}
  \end{equation}
Set $x_i=x+i(y-x)/n$ for $i=0,1,\ldots, n$. Then $x_0=x$ and
$x_n=y$; moreover, \eqref{proof.2} implies
$|x_{i-1}-x_{i}|=|x-y|/n\le \eta/C_6$ for all $i=1,2,\ldots,n$.
Therefore, for all $p\ge 1$, by \eqref{small-1},
  \begin{align*}
  W_{p}(\delta_x P_t,\delta_y P_t)&\leq \sum_{i=1}^n W_{p}(\delta_{x_{i-1}} P_t,\delta_{x_{i}} P_t)\\
  &\leq  C_7^{1/p} e^{-{\lambda}t/p}\sum_{i=1}^n |x_{i-1}-x_{i}|^{1/p}\\
  &\leq C_7^{1/p} e^{-{\lambda}t/p}\, n(\eta/C_6)^{1/p}  \\
  &\leq C_8e^{-{\lambda t}/p}|x-y|,
  \end{align*}
where in the last inequality we have used $n\leq 2C_6|x-y|/\eta$. The proof of Theorem \ref{main-result} is completed.
\end{proof}

Next, we turn to the

\begin{proof}[Proof of Theorem $\ref{thm-theta}$]
Since we assume $\eta>0$, condition \eqref{thm-theta-1} implies that  \eqref{main-result-1} holds with $c$ replaced by $c\eta^{\theta-1}$.  Then, we can directly apply the assertion of Theorem \ref{main-result} to conclude that there exists a constant $\lambda$ (which is given in Theorem \ref{main-result} with $c$ replaced by $c\eta^{\theta-1}$) such that for all $p\ge1$ and $x,y\in\R^d$
  \begin{equation}\label{proof-theta-1}
  W_p(\delta_x P_t,\delta_y P_t)\leq Ce^{-\lambda t/p} (|x-y|^{1/p}\vee |x-y|).
  \end{equation}

To complete the proof, we only need to consider the case that $x,y\in\R^d$ with $|\sigma^{-1}(x-y)|>\eta$ and $t>0$ large enough. For this, we use both the synchronous coupling
and the coupling by reflection defined by \eqref{coupling}. In
details, with \eqref{Ito-SDE}, we now consider
  \begin{equation*}\label{coupling-1}
  \d Y_t=\begin{cases}\sigma\, \d B_t+b(Y_t)\,\d t, & 0\leq t< T_{\eta},\\
  \sigma(\Id-2e_te_t^\ast)\,\d B_t+b(Y_t)\,\d t, & T_{\eta}\le t< T,
  \end{cases}
  \end{equation*}
where
  $$T_{\eta}=\inf\{t>0:|\sigma^{-1}(X_t-Y_t)|=\eta\}$$
and $T=\inf\{t>0:X_t=Y_t\}$ is the coupling time. For $t\geq T$, we
still set $Y_t=X_t$. Therefore, the difference process
$(Z_t)_{t\ge0}=(X_t-Y_t)_{t\ge0}$ satisfies
  \begin{equation*}\label{difference-1}
  \d Z_t=(b(X_t)-b(Y_t))\,\d t,\quad t<T_{\eta}.
  \end{equation*}
As a result,
  $$\d|\sigma^{-1}Z_t|^2=2\big\<\sigma^{-1}Z_t,\sigma^{-1}(b(X_t)-b(Y_t))\big\>\,\d t.$$
Still denoting by $r_t=|\sigma^{-1}Z_t|$, we get
  $$\d r_t\leq 2\kappa(r_t)\,\d t\le -2cr_t^\theta \,\d t,\quad t<T_{\eta},$$
which implies that
  \begin{equation}\label{rrr2}
  T_{\eta}\le\frac{1}{2c(1-\theta)}\left(|\sigma^{-1}(x-y)|^{1-\theta}-\eta^{1-\theta}\right)
  \le \frac{\eta^{1-\theta}}{2c(\theta-1)}=: t_0
  \end{equation}
since $\theta>1$.

Therefore, for any $x,y\in\R^d$ with $|\sigma^{-1}(x-y)|> \eta$,
$p\ge1$ and $t>t_0$, we have
  \begin{equation*}
  \begin{split}
  \E|\sigma^{-1}(X_t-Y_t)|^p& =\E\big[\E^{(X_{T_{\eta}}, Y_{T_{\eta}})}|\sigma^{-1}(X_{t-T_{\eta}}-Y_{t-T_{\eta}})|^p\big]\\
  & \le C_9\E\big[|\sigma^{-1}(X_{T_{\eta}}-Y_{T_{\eta}})| e^{-\lambda (t-T_{\eta})}\big]\\
  &\le C_9\eta e^{\lambda t_0} e^{-\lambda t},
  \end{split}
  \end{equation*}
where in the first inequality we have used \eqref{small}, and the
last inequality follows from \eqref{rrr2}. In particular, we have
for all  $|\sigma^{-1}(x-y)|>\eta$ and $t>t_0$,
  $$\E|X_t-Y_t|^p\le C_{10} e^{-\lambda t}$$ and so
  $$ W_{p}(\delta_x P_t,\delta_y P_t)\le C_{10}e^{-\lambda t}.$$
Combining with all conclusions above, we complete the proof of Theorem \ref{thm-theta}.
\end{proof}

Finally, we present the

\begin{proof}[Proof of  Proposition $\ref{2-prop-1}$]
By the definition of $\kappa$, for all $x,y\in\R^d$ with $|\sigma^{-1}(x-y)|=r_0$, we have
  \begin{equation}\label{2-prop-1.1}
  \frac{\big\<\sigma^{-1}(x-y),\sigma^{-1}(b(x)-b(y))\big\>}{|\sigma^{-1}(x-y)|}\leq -2c.
  \end{equation}
For any fixed $x,y\in\R^d$ with $r=|\sigma^{-1}(x-y)|$ large enough,
let $n_0=[r/r_0]$ be the integer part of $r/r_0$. Denote by $x_0=x$
and $x_{n_0+1}=y$. We can find $n_0$ points
$\{x_1,x_2,\ldots,x_{n_0}\}$ on the line segment linking $x$ to $y$,
such that $|\sigma^{-1}(x_{i-1}-x_i)|=r_0$ for $i=1,2,\ldots, n_0$
and $|\sigma^{-1}(x_{n_0}-x_{n_0+1})|=|\sigma^{-1}(x_{n_0}-y)|\le r_0$.
Then
  \begin{align*}
  &\hskip14pt \frac{\big\<\sigma^{-1}(x-y),\sigma^{-1}(b(x)-b(y))\big\>}{|\sigma^{-1}(x-y)|}\\
  &=\sum_{i=1}^{n_0}\frac{\big\<\sigma^{-1}(x-y),\sigma^{-1}(b(x_{i-1})-b(x_i))\big\>}{|\sigma^{-1}(x-y)|}
  +\frac{\big\<\sigma^{-1}(x-y),\sigma^{-1}(b(x_{n_0})-b(y))\big\>}{|\sigma^{-1}(x-y)|}\\
  &= \sum_{i=1}^{n_0}\frac{\big\<\sigma^{-1}(x_{i-1}-x_{i}),\sigma^{-1}(b(x_{i-1})-b(x_i))\big\>}
  {|\sigma^{-1}(x_{i-1}-x_{i})|}
  +\frac{\big\<\sigma^{-1}(x_{n_0}-y),\sigma^{-1}(b(x_{n_0})-b(y))\big\>}{|\sigma^{-1}(x_{n_0}-y)|}.
  \end{align*}
By \eqref{2-prop-1.1} and our assumption on $b$,
  $$\frac{\big\<\sigma^{-1}(x-y),\sigma^{-1}(b(x)-b(y))\big\>}{|\sigma^{-1}(x-y)|}
  \leq -2cn_0 +\delta_{0}.$$
Next, since ${r}/{r_0}\leq n_0+1$, we have
  $$-2c \frac{r}{r_0}\geq -2cn_0-2c$$
which implies $-2cn_0\leq 2c -2c{r}/{r_0}$. Therefore
  $$\frac{\big\<\sigma^{-1}(x-y),\sigma^{-1}(b(x)-b(y))\big\>}{|\sigma^{-1}(x-y)|}
  \leq \delta_{0}+2c-\frac{2c}{r_0} r$$
for all $x,y\in\R^d$ with $r=|\sigma^{-1}(x-y)|$. As a result, the definition of $\kappa(r)$ leads to
  \begin{align*}
  \kappa(r)\leq \frac12 \delta_{0}+c-\frac{c}{r_0} r
  \leq -\frac{c}{2r_0} r
  \end{align*}
for all $r\geq r_0(\delta_{0}+2c)/c$, and so \eqref{diss-2} holds  with the new constant $c/2r_0$.
\end{proof}

\section{Proofs of Examples and Corollaries}
\begin{proof}[Proof of Example $\ref{1-example-1}$]
(1) Since $\sigma=\Id$, the supremum in the definition of
$\kappa(r)$ is taken over all $x,y\in\R^d$ with $|x-y|=r$. Thus, to
verify $\kappa(r)\geq 0$ for $r>0$ large enough, it suffices to show
that the supremum is nonnegative when $x,y$ are restricted on one of
the coordinate axes with $r=|x-y|$ large enough, that is, we can
assume the dimension is 1. Then
  $$V(x)=-(1+x^2)^{\delta/2},\quad \delta\in(0,1), x\in\R.$$
Now the result follows immediately from the fact that $V'(x)$ is
strictly increasing when $|x|$ is large enough. Indeed, we have
  $$V''(x)=\delta(1+x^2)^{\frac\delta 2-2}[(1-\delta)x^2-1]$$
which is positive if $|x|\geq (1-\delta)^{-1/2}$.

On the other hand, it is easy to see that with the choices of
$\sigma$ and $b$ above, the semigroup $(P_t)_{t\ge0}$ is symmetric
with respect to the probability measure $\mu(\d x)=\frac{1}{Z_V}
e^{V(x)}\,\d x$. Then, according to (the proof of) Corollary
\ref{poincare} below, we know that $\mu(\d x)$ fulfills the
Poincar\'{e} inequality \eqref{poincare-2} if \eqref{main-result-2}
is satisfied with $p=1$; however, this is impossible, see e.g.
\cite[Example 4.3.1 (3)]{Wangbook}.

(2) In this case, $V(x)=-(1+|x|^2)^{\delta/2}$ with
$\delta\in[1,2)$. First, we prove that $V$ is strictly concave on
$\R^d$. Indeed, for all $1\leq i,j\leq d$,
  $$\frac{\partial^2 V}{\partial x_i\partial x_j}(x)=\delta(1+|x|^2)^{\delta/2-2}
  \big[(2-\delta) x_ix_j-\delta_{ij}(1+|x|^2)\big],\quad x\in\R^d.$$
Therefore, for any $z\in\R^d$ with $z\neq 0$,
  \begin{align*}
  \sum_{i,j=1}^d \frac{\partial^2 V}{\partial x_i\partial x_j}(x)z_i z_j
  &=\delta(1+|x|^2)^{\delta/2-2}\sum_{i,j=1}^d\big[(2-\delta)x_ix_j-\delta_{ij}(1+|x|^2)\big]z_i z_j\\
  &=\delta(1+|x|^2)^{\delta/2-2}\big[(2-\delta)\<x,z\>^2-(1+|x|^2)|z|^2\big]\\
  &\leq -\delta(1+|x|^2)^{\delta/2-2}|z|^2<0,
  \end{align*}
which implies $V(x)$ is strictly concave. Hence $\kappa(r)\leq 0$
for all $r\geq 0$.  On the other hand, to show that $\kappa(r)\ge0$
for all $r\geq 0$, as in the proof of (1), we simply look at the one
dimensional case: $$V(x)=-(1+x^2)^{\delta/2},\quad  \delta\in[1,2),
x\in\R.$$ For any fixed $r>0$ and for any $x>0$,
  \begin{align*}
  \kappa(r)&\geq \frac12 (V'(x+r)-V'(x))\\
  &\ge \frac r2 \inf_{x\le s\le x+r} V''(s)\\
  &=\frac r2 \inf_{ x\le s\le x+r}\frac{-\delta[(\delta-1)s^2+1]}{(1+s^2)^{2-\delta/2}},
  \end{align*}
which implies that
  $$\kappa(r)\ge \frac r2\lim_{x\to \infty}
  \inf_{ x\le s\le x+r}\frac{-\delta[(\delta-1)s^2+1]}{(1+s^2)^{2-\delta/2}}=0 .$$
Therefore, $\kappa(r)=0$ for all $r>0$.

We have seen from above that for $V(x)=-(1+|x|^2)^{\delta/2}$ with
$\delta\in[1,2)$, $\kappa(r)=0$ for all $r\geq 0$, thus for any
$x,y\in\R^d$,
  $$\langle b(x)-b(y),x-y\rangle\le 0.$$
Then, the assertion \eqref{1-example-1.1} immediately follows from
(the proof of) Theorem \ref{theorem1}, by simply using the
synchronous coupling, see e.g.\ \cite[p.2432]{BGG}.

Finally we prove the algebraic convergence rate
\eqref{1-example-1.2}. For this, we mainly follow from \cite[Section
5]{CL} or \cite[Section 7.2]{CG} (see also \cite[Theorem 1.1]{PW}).
By \eqref{difference.1},
  $$\d r_t\leq 2\,\d W_t,\quad t< T,$$
where $T$ is the coupling time of the coupling process
$(X_t,Y_t)_{t\ge0}$, and $(W_t)_{t\ge0}$ is the same one-dimensional
Brownian motion as in \eqref{difference}. Hence,
  $$r_t\leq |x-y|+2W_t,\quad  t<T.$$
Let
  $$\tau_z:=\inf \{t>0: W_t=z \}.$$
Then
  $$T\leq \tau_{-|x-y|/2}.$$
Denote by $W_t^\ast=\inf_{0\leq s\leq t}W_s$ which has the same law
as that of $-|W_t|$. Thus for any $t>0$,
  \begin{align*}
  \P(r_t>0)&= \P(T>t)\le \P \big(\tau_{-|x-y|/2}>t\big)\\
  &=\P \big(W_t^\ast>-|x-y|/2\big)=\P \big(-|W_t|>-|x-y|/2\big)\\
  &=2\int_0^{|x-y|/2} \frac1{\sqrt{2\pi t}} e^{-s^2/2t}\,\d s \\
  &\leq \frac{|x-y|}{\sqrt{2\pi t}}.
  \end{align*}
Therefore, for any $f\in C_b(\R^d)$ with $\|f\|_\infty\leq 1$, we
have
  \begin{align*}
  \big|\E(f(X_t)-f(Y_t))\big|&=\big|\E\big[(f(X_t)-f(Y_t))\ch_{\{r_t>0\}}\big]\big|\\
  &\leq 2\,\P(r_t>0) \leq \sqrt{\frac{2}{\pi t}}\, |x-y|.
  \end{align*}
In particular, by the definition of total variation distance,
  \begin{align*}
  d_{TV}(\delta_x P_t,\delta_y P_t)&= \sup\big\{\big|\E(f(X_t)-f(Y_t))\big|:f\in C_b(\R^d),\|f\|_\infty\leq 1\big\}
  \leq \sqrt{\frac{2}{\pi t}}\, |x-y|.
  \end{align*}
The proof is complete.
\end{proof}

Finally we present the proofs of the two corollaries of Theorem \ref{main-result}.

\begin{proof}[Proof of Corollary $\ref{invariant}$]
Recall that for all $p\in[1,\infty)$, $\mathcal{P}_p(\R^d)$ is the
space of all probability measures $\nu$ on $(\R^d,
\mathcal{B}(\R^d))$ satisfying $\int |z|^p\,\nu(\d z)<\infty.$  Note
that, we assume the solution $(X_t)_{t\ge0}$ to \eqref{Ito-SDE} has
finite moments of all orders. In particular, for any $x\in\R^d$,
$t>0$ and $p\ge1$, $\delta_x P_t\in \mathcal{P}_p(\R^d)$. According
to \eqref{main-result-1} and Theorem \ref{main-result}, there is a constant $\lambda>0$ such that for any
$p\in[1,\infty)$ and any $\nu_1$, $\nu_2\in \mathcal{P}_p(\R^d)$,
  \begin{equation}
  \label{invariant1}W_p(\nu_1P_t,\nu_2P_t)\le C_pe^{-\lambda t}W_{\phi_p}(\nu_1,\nu_2),\quad t>0,
  \end{equation}
where $C_p$ is a positive constant. In particular, for any $\nu_1$, $\nu_2\in \mathcal{P}_1(\R^d)$,
  $$W_1(\nu_1P_t,\nu_2P_t)\le C_1e^{-\lambda t}W_1(\nu_1,\nu_2),\quad t>0.$$

Let $t_0>0$ such that $C_1e^{-\lambda_1 t_0}<1$. Then, the map $\nu
\mapsto \nu P_{t_0}$ is a contraction on the complete metric space
$(\mathcal{P}_1(\R^d), W_1)$. Hence, by the Banach fixed point theorem,
there exists a unique probability measure $\mu_{t_0}$ such that
$\mu_{t_0}P_{t_0}=\mu_{t_0}$. Let
$\mu:=t_0^{-1}\int_0^{t_0}\mu_{t_0}P_s\,\d s$. It is easy to see
that $\mu P_t=\mu$ for all $t\in[0,t_0]$ and so for all
$t\in[0,\infty)$. Therefore, $\mu$ is a invariant probability for
the semigroup $(P_t)_{t\ge0}$. Moreover, for any $\nu\in
\mathcal{P}_1(\R^d)$ and $t>0$,
  $$W_1(\nu P_t,\mu)=W_1(\nu P_t,\mu P_t)\le C_1e^{-\lambda t}W_1(\nu,\mu).$$
The inequality above also yields the uniqueness of the invariant measure.

On the other hand, since $b$ is locally bounded and satisfies
\eqref{main-result-1}, it follows from \cite[Theorem 4.3 (ii)]{MT93}
that the unique invariant measure $\mu\in \cap_{p\geq 1}\mathcal{P}_p(\R^d)$.
Now, replacing $\nu_2$ with $\mu$ in \eqref{invariant1}, we arrive
at
$$W_p(\nu_1 P_t,\mu)=W_p(\nu_1 P_t,\mu P_t)\le C_pe^{-\lambda t}W_{\phi_p}(\nu_1,\mu),\quad
t>0$$ for any $\nu_1\in \mathcal{P}_{p}(\R^d)$. The proof is completed.
\end{proof}

\begin{proof}[Proof of Corollary $\ref{poincare}$]
Let $(P_t)_{t\ge0}$ be the semigroup generated by $L=\Delta-\nabla
U\cdot\nabla$. Then $(P_t)_{t\ge0}$ is symmetric with respect to the
probability measure $\mu$. Since the Hessian matrix $\textup{Hess}(U)\geq -K$,
we deduce that $\kappa(r)\leq Kr/2$ for all $r>0$. Moreover, replacing $b$ by $-\nabla U$ in the definition of $\kappa(r)$, we deduce from \eqref{poincare-1} that $\kappa(L)<0$. According to Proposition
\ref{2-prop-1} and Theorem \ref{main-result}, we know that there exist two positive constants $C$, $\lambda$
such that for all $t>0$ and $x$, $y\in\R^d$,
  $$W_1(\delta_x P_t,\delta_y P_t)\le Ce^{-\lambda t}|x-y|.$$
This implies that (e.g.\ see \cite[Theorem 5.10]{Chenbook} or \cite[Theorem 5.10]{Villani})
  \begin{equation}\label{lip}
  \|P_tf\|_{\textrm{Lip}}\le Ce^{-\lambda t} \|f\|_{\textrm{Lip}}
  \end{equation}
holds for any $t>0$ and any Lipschitz continuous function $f$, where $\|f\|_{\textrm{Lip}}$ denotes the Lipschitz semi-norm with respect to the Euclidean norm $|\cdot|$.

On the other hand,  for any $f\in C_c^2(\R^d)$,  by \eqref{lip}, we have
  \begin{align*}\textrm{Var}_{\mu}(f) &=\mu(f^2)-\mu(f)^2=-\int_{\R^d} \int_0^\infty \partial_t (P_tf)^2\,\d t\,\d \mu\\
  &=-2\int_0^\infty \int_{\R^d} P_t f L P_t f\,\d \mu\,\d t\\
  &= \int_0^\infty \int_{\R^d} |\nabla P_t f|^2\,\d \mu\,\d t\\
  &\le  \int_0^\infty \| P_t f\|_{\textrm{Lip}}^2\,\d t\le \frac{C^2}{2\lambda}\|f\|_{\textrm{Lip}}^2.\end{align*}
Replacing $f$ with $P_tf$ in the equality above, we arrive at
  $$ \textrm{Var}_{\mu}(P_tf)\le \frac{C^2}{2\lambda}\|P_tf\|_{\textrm{Lip}}^2 \le \frac{C^4e^{-2\lambda t}}{2\lambda}\|f\|_{\textrm{Lip}}^2.$$

Next, we follow the proof of \cite[Lemma 2.2]{R-Wang01} to show that the inequality above yields the desired Poincar\'{e} inequality. Indeed, for every $f$ with $\mu(f)=0$ and $\mu(f^2)=1$. By the spectral representation theorem, we have
  \begin{align*} \|P_tf\|_{L^2(\mu)}^2&=\int_0^\infty e^{-2ut}\,\d (E_uf,f)\\
  &\ge \left[\int_0^\infty e^{-2us}\, \d (E_uf,f)\right]^{t/s}\\
  &= \|P_sf\|_{L^2(\mu)}^{2t/s},\quad t\ge s,
  \end{align*}
where in the inequality above we have used Jensen's inequality. Thus,
  $$\|P_sf\|_{L^2(\mu)}^2\le\bigg[\frac{C^4e^{-2\lambda t}}{2\lambda}\|f\|_{\textrm{Lip}}^2\bigg]^{s/t}\le \frac{C^{4s/t}}{(2\lambda)^{s/t}}\|f\|_{\textrm{Lip}}^{2s/t} e^{-2\lambda s}. $$
Letting $t\to\infty$, we get that
  $$\|P_sf\|_{L^2(\mu)}^2\le e^{-2\lambda s}, $$
which is equivalent to the desired Poincar\'{e} inequality, see e.g.\ \cite[Theorem 1.1.1]{Wangbook}.
\end{proof}

\ \

\noindent{\bf Acknowledgements.} The authors are grateful to the
referee for his valuable comments which helped to improve the
quality of the paper.


\begin{thebibliography}{99}
\bibitem{BCG} Bakry, D., Cattiaux, P. and Guillin, A.: Rate of convergence for ergodic continuous Markov
processes: Lyapunov versus Poincar\'e. \emph{J. Funct. Anal.} \textbf{254} (2008), no. 3, 727--759.

\bibitem{BGG} Bolley, F., Gentil, I. and Guillin, A.: Convergence to equilibrium in Wasserstein distance for
Fokker--Planck equations. \emph{J. Funct. Anal.} \textbf{263} (2012), no. 8, 2430--2457.

\bibitem{BGG13} Bolley, F., Gentil, I. and Guillin, A.: Uniform convergence to equilibrium for granular media.
\emph{Arch. Ration. Mech. Anal.} \textbf{208} (2013), no. 2, 429--445.


\bibitem{Bu}  Butkovsky, O.: Subgeometric rates of convergence of Markov processes in Wasserstein metric.
\emph{Ann. Appl. Probab.} \textbf{24} (2014), no. 2, 526--552.

\bibitem{CG} Cattiaux, P. and Guillin, A.: Semi log-concave Markov diffusions,
{S\'eminaire de Probabilit\'es XLVI}, Lecture Notes in Mathematics {\bf 2123}, Springer Verlag, 2015, 231--292.

\bibitem{Chenbook} Chen, M.-F.: From Markov Chains to Non-Equilibrium Particle Systems. Second edition.
\emph{World Scientific Publishing Co., Inc., River Edge, NJ}, 2004.

\bibitem{Chen} Chen, M.-F.: Eigenvalues, Inequalties, and Ergodic Theory.
Probability and its Applications (New York). \emph{Springer--Verlag London, Ltd., London}, 2005.

\bibitem{Chen12} Chen, M.-F.: Basic estimates of stability rate for one-dimensional diffusions,
Probability Approximations and Beyond, Lecture Notes in Statistics, 2012, 75--99.

\bibitem{CL}  Chen, M.-F. and Li, S.: Coupling methods for multidimensional diffusion processes.
\emph{Ann. Probab.} \textbf{17} (1989), no. 1, 151--177.

\bibitem{Chen-Wang-97}  Chen, M.-F. and Wang, F.-Y.: Estimation of spectral gap for elliptic operators.
\emph{Trans. Amer. Math. Soc.} \textbf{249} (1997), no. 3, 1239--1267.

\bibitem{Lutz} D\"umbgen, L: Bounding standard Gaussian tail probabilities, arXiv:1012.2063v3.

\bibitem{Eberle11} Eberle, A.: Reflection coupling and Wasserstein contractivity without convexity.
\emph{C. R. Math. Acad. Sci. Paris} \textbf{349} (2011), no. 19--20, 1101--1104.

\bibitem{Eberle} Eberle, A.:\ Reflection couplings and contraction rates for
diffusions,\
arXiv:1305.1233v3.

\bibitem{FGP10a} Flandoli, F., Gubinelli, M. and Priola, E.: Well-posedness of the transport
equation by stochastic perturbation. \emph{Invent. Math.} \textbf{180} (2010), no. 1, 1--53.

\bibitem{Fort05} Fort, G. and Roberts, G.O.: Subgeometric ergodicity of strong Markov processes.
\emph{Ann. Appl. Probab.} \textbf{15} (2005), no. 2, 1565--1589.


\bibitem{Hairer} Hairer, M. and Mattingly, J.C.: Spectral gaps in Wasserstein distances and the
2D stochastic Navier--Stokes equations. \emph{Ann. Probab.} \textbf{36} (2008), no. 6, 2050--2091.

\bibitem{HM2} Hairer, M., Mattingly, J.C. and Scheutzow, M.: Asymptotic coupling and a general form
of Harris's theorem with applications to stochastic delay equations. \emph{Probab. Theory Related
Fields} \textbf{149} (2011), no. 1--2, 223--259.

\bibitem{LR}  Lindvall, T. and Rogers, L.: Coupling of multidimensional diffusions by reflection.
\emph{Ann. Probab.} \textbf{14} (1986), no. 3, 860--872.

\bibitem{Meyn} Meyn, S.P. and Tweedie, R.L.: Markov Chains and Stochastic Stability. Communications and
Control Engineering Series. \emph{Springer--Verlag London, Ltd., London}, 1993.

\bibitem{MT93} Meyn, S.P. and Tweedie, R.L.: Stability of Markovian processes (III): Foster-Lyapunov
criteria for continuous-time processes. \emph{Adv. Appl. Probab.} \textbf{25} (1993), no. 3, 518--548.

\bibitem{PW} Priola, E. and Wang, F.-Y.: Gradient estimates for diffusion semigroups with singular coefficients. \emph{J. Funct. Anal.}
\textbf{236} (2006), no. 1, 244--264.

\bibitem{RS} von Renesse, M. and Sturm, K.: Transport inequalities, gradient estimates, entropy,
and Ricci curvature. \emph{Comm. Pure Appl. Math.} \textbf{58} (2005), no. 7, 923--940.

\bibitem{R-Wang01} R\"{o}ckner, M. and Wang, F.-Y.: Weak Poincar\'{e} inequalities and $L^2$-convergence
rates of Markov smigroups. \emph{J. Funct. Anal.} \textbf{185} (2001), no. 2, 564--603.

\bibitem{Villani} Villani, C.: Optimal Transport: Old and New. Grundlehren der Mathematischen Wissenschaften
[Fundamental Principles of Mathematical Sciences], 338. \emph{Springer--Verlag, Berlin}, 2009.

\bibitem{Wang99} Wang, F.-Y.: Existence of the spectral gap for elliptic operators. \emph{Ark. Math.}
\textbf{37} (1999), no. 2, 395--407.

\bibitem{Wang00} Wang, F.-Y.: Functional inequalities for empty essential spectrum. \emph{J. Funct. Anal.}
\textbf{170} (2000), no. 1, 219--245.

\bibitem{Wangbook} Wang, F.-Y.: Functional Inequalities, Markov Semigroups and Spectral Theory. \emph{Science Press, Beijing,} 2005.

\bibitem{Wang2014} Wang, F.-Y.:  Gradient estimates and applications for SDEs in Hilbert space with multiplicative noise and Dini continuous drift, arXiv:1404.2990v4.


\end{thebibliography}
\end{document}